\documentclass[12pt]{article}
\usepackage{amssymb}
\usepackage{amsthm}
\usepackage{amsmath}
\usepackage{graphicx}
\usepackage{enumerate}
\usepackage{pict2e}
\usepackage{framed}

\DeclareMathOperator{\Tr}{Tr}

\newtheorem{theorem}{Theorem}[section]
\newtheorem{definition}[theorem]{Definition}
\newtheorem{lemma}[theorem]{Lemma}
\newtheorem{proposition}[theorem]{Proposition}

\newtheorem{example}[theorem]{Example}

\title{Lattices and semilattices derived from commutative rings of characteristic $2$ satisfying the identity $x^{2^n}\approx x$\thanks{Support of the research of the first author by the Czech Science Foundation (GA\v CR), project 25-20013L, and support of the research of the third author by the Austrian Science Fund (FWF), project 10.55776/PIN5424624, is gratefully acknowledged.}}
\author{Ivan~Chajda, Miroslav~Kola\v r\'ik and Helmut~L\"anger}
\date{}

\begin{document}
	
\maketitle
	
\begin{abstract}
We prove that a commutative ring $\mathbf R=(R,+,\cdot)$ of characteristic $2$ satisfying the identity $x^{2^n}\approx x$ together with the binary relation $\le$ on $R$ defined by $x\le y$ if $xy=x^2$ forms a meet-semilattice with smallest element $0$. If, moreover, $\mathbf R$ is unitary then we derive two binary term operations $\wedge$ and $\vee$ on $R$ which together with the unary term operation $x':=x+1$ form a Boolean algebra.
\end{abstract}

{\bf AMS Subject Classification:} 16R10, 16R50, 06E20

{\bf Keywords:} commutative ring, finite field, polynomial identity, term operation, Boolean algebra, semilattice

\section{Introduction}

In his pioneering works \cite{Bo47} and \cite{Bo54} G.~Boole showed that the classical propositional calculus can be formalized by unitary Boolean rings. Later on, G.~Birkhoff \cite{Bi} showed that to every such ring there can be assigned a Boolean algebra, i.e.\ a bounded distributive lattice with complementation. Also vice versa, to every Boolean algebra there can be assigned a unitary Boolean ring and these assignments are in a one-to-one correspondence. Recall that a Boolean ring is a ring satisfying the identity $x^2\approx x$. It is easy to show that every such ring is commutative and of characteristic $2$, i.e.\ it satisfies the identity $x+x\approx0$. Surprisingly, in \cite{KDG} it was proved that a Boolean algebra can be derived also from a unitary ring satisfying the identity $x^4\approx x$. In their recent paper \cite{CK}, the first two authors obtained a similar result for rings satisfying the identity $x^{2^n}\approx x$ for $n\in\{3,4\}$, but neither the identity $x^2\approx x$ nor the identity $x^4\approx x$. In the present paper we generalize this to the case of arbitrary positive integers $n$.

\section{The general construction}

In the following let $n$ be a positive integer and put $q:=2^n$.

A {\em ring} $(R,+,\cdot)$ satisfying the identity $x+x\approx0$ is said to be {\em of characteristic $2$}. A commutative binary operation $\wedge$ on $R$ is called {\em distributive with respect to $+$} if $(x+y)\wedge z=(x\wedge z)+(y\wedge z)$ for all $x,y,z\in R$. Let $\mathbb N_0$ denote the set of all non-negative integers. It is well-known and easy to see that in a commutative ring $(R,+,\cdot)$ of characteristic $2$ we have
\[
(x+y)^{2^k}=x^{2^k}+y^{2^k}
\]
for all $x,y\in R$ and all $k\in\mathbb N_0$.

\begin{theorem}\label{th1}
Let $\mathbf R=(R,+,\cdot)$ be a ring of characteristic $2$, $\wedge$ an idempotent, commutative and associative binary operation on $R$, assume $\wedge$ to be distributive with respect to $+$ and define $x\vee y:=x+y+(x\wedge y)$ for all $x,y\in R$. Then the following hold:
\begin{enumerate}[{\rm(i)}]
\item $(R,\vee,\wedge)$ is a distributive lattice,
\item if $\mathbf R$ is unitary and satisfies the identity $x\wedge1\approx x$ and $x':=x+1$ for all $x\in R$ then $(R,\vee,\wedge,{}',0,1)$ is a Boolean algebra.
\end{enumerate}
\end{theorem}

\begin{proof}
Let $a,b,c\in R$.
\begin{enumerate}[(i)]
\item We have
\begin{align*}
          a\vee b & =a+b+(a\wedge b)=b+a+(b\wedge a)=b\vee a, \\
  (a\vee b)\vee c & =a+b+(a\wedge b)+c+\Big(\big(a+b+(a\wedge b)\big)\wedge c\Big)= \\
                  & =a+b+c+(a\wedge b)+(a\wedge c)+(b\wedge c)+(a\wedge b\wedge c)= \\
                  & =a+b+c+(b\wedge c)+\Big(a\wedge\big(b+c+(b\wedge c)\big)\Big)=a\vee(b\vee c), \\
(a\vee b)\wedge a & =\big(a+b+(a\wedge b)\big)\wedge a=a+(a\wedge b)+(a\wedge b)=a, \\
(a\wedge b)\vee a & =(a\wedge b)+a+\big((a\wedge b)\wedge a\big)=(a\wedge b)+a+(a\wedge b)=a, \\
(a\vee b)\wedge c & =\big(a+b+(a\wedge b)\big)\wedge c=(a\wedge c)+(b\wedge c)+(a\wedge b\wedge c)= \\
                  & =(a\wedge c)+(b\wedge c)+\big((a\wedge c)\wedge(b\wedge c)\big)=(a\wedge c)\vee(b\wedge c).        
\end{align*}
\item
We have
\begin{align*}
  a\vee(a+1) & =a+a+1+\big(a\wedge(a+1)\big)=1+(a\wedge a)+(a\wedge1)=1+a+a=1, \\
a\wedge(a+1) & =(a\wedge a)+(a\wedge1)=a+a=0.	
\end{align*}
\end{enumerate}	
\end{proof}

Observe that every unitary ring satisfying the identity $x^{2m}\approx x$ for a positive integer $m$ is of characteristic $2$. This can be seen as follows: Let $(R,+,\cdot,1)$ be such a ring. Then $-1=(-1)^{2m}=1$ and hence $1+1=0$ which shows $x+x=0$ for all $x\in R$.

\section{Derived semilattice}

In the following let $\mathbf R=(R,+,\cdot)$ be a (not necessarily unitary) commutative ring of characteristic $2$ satisfying the identity $x^{2^n}\approx x$ where $n$ is a positive integer.

For $a,b\in R$ we define $a^0b:=b$.

\begin{lemma}
Let $a,b\in R$ and $k$ be a positive integer. Then
\[
(a+b)^{2^k-1}=\sum_{i=0}^{2^k-1}a^ib^{2^k-1-i}.\tag{1}
\]	
\end{lemma}

\begin{proof}
We use induction on $k$. Obviously, (1) holds for $k=1$. Now assume that $k$ is a positive integer and (1) holds for $k$. Then
\begin{align*}
(a+b)^{2^{k+1}-1} & =\big((a+b)^{2^k-1}\big)^2(a+b)=\left(\sum_{i=0}^{2^k-1}a^ib^{2^k-1-i}\right)^2(a+b)= \\
                  & =\left(\sum_{i=0}^{2^k-1}a^{2i}b^{2^{k+1}-2-2i}\right)(a+b)=\sum_{i=0}^{2^k-1}a^{2i+1}b^{2^{k+1}-2-2i}+\sum_{i=0}^{2^k-1}a^{2i}b^{2^{k+1}-1-2i}= \\
                  & =\sum_{j=0}^{2^{k+1}-1}a^jb^{2^{k+1}-1-j},
\end{align*}
i.e.\ (1) holds for $k+1$ instead of $k$.
\end{proof}

Now we define the main concept of this section.

\begin{definition}\label{def1}
On $R$ we define a binary term operation $\wedge$ as follows:
\[
x\wedge y:=\sum_{i=1}^{q-1}x^iy^{q-i}
\]
for all $x,y\in R$.
\end{definition}

Basic properties of this term operation are listed in the following Proposition.

\begin{proposition}\label{prop2}
Let $a,b\in R$ and $k$ be a positive integer. Then the following hold:
\begin{enumerate}[{\rm(i)}]
\item $a\wedge a=a$, $a\wedge b=b\wedge a$ and $a\wedge0=0$,
\item $a\wedge b=a+a(a+b)^{q-1}=b+b(a+b)^{q-1}$,
\item $a^{q-1}(a\wedge b)=a\wedge b$,
\item $(a\wedge b)^2=a^2\wedge b^2$,
\item $(a+b)(a\wedge b)=0$,
\item $(a\wedge b)^k=a^k+a^k(a+b)^{q-1}=b^k+b^k(a+b)^{q-1}$.
\end{enumerate}
\end{proposition}

\begin{proof}
$\text{}$
\begin{enumerate}[(i)]
\item This is immediate.
\item We have
\begin{align*}
a+a(a+b)^{q-1} & =a+a\sum_{i=0}^{q-1}a^ib^{q-1-i}=a+\sum_{i=0}^{q-1}a^{i+1}b^{q-1-i}=a+\sum_{j=1}^qa^jb^{q-j}= \\
         	   & =a+\sum_{j=1}^{q-1}a^jb^{q-j}+a^q=\sum_{j=1}^{q-1}a^jb^{q-j}=a\wedge b.
\end{align*}
The last formula follows by symmetry of $a$ and $b$.	
\item Using (ii) we obtain
\begin{align*}
a^{q-1}(a\wedge b) & =a^{q-1}\big(a+a(a+b)^{q-1}\big)=a^q+a^q(a+b)^{q-1}=a+a(a+b)^{q-1}= \\
	               & =a\wedge b.	
\end{align*}
\item Using (ii) we obtain
\begin{align*}
(a\wedge b)^2 & =\big(a+a(a+b)^{q-1}\big)^2=a^2+a^2(a+b)^{2^{n+1}-2}=a^2+a^2(a^2+b^2)^{q-1}= \\
	          & =a^2\wedge b^2.	
\end{align*}
\item Using (ii) we obtain
\begin{align*}
(a+b)(a\wedge b) & =(a+b)\big(a+a(a+b)^{q-1}\big)=(a+b)a+a(a+b)^q= \\
	             & =a(a+b)+a(a+b)=0.	
\end{align*}
\item We use induction on $k$. The case $k=1$ follows from (ii). Now assume that $k$ is a positive integer and (vi) holds for $k$. Then again by (ii)
\begin{align*}
(a\wedge b)^{k+1} & =(a\wedge b)^k(a\wedge b)=\big(a^k+a^k(a+b)^{q-1}\big)\big(a+a(a+b)^{q-1}\big)= \\
                  & =a^{k+1}+a^{k+1}(a+b)^{q-1}+a^{k+1}(a+b)^{q-1}+a^{k+1}(a+b)^{2q-2}= \\
                  & =a^{k+1}+a^{k+1}(a+b)^{q-1}.
\end{align*}
The last formula follows by symmetry of $a$ and $b$.	
\end{enumerate}
\end{proof}	

Now we are able to state the main result of this section.

\begin{theorem}
The groupoid $(R,\wedge)$ is a meet-semilattice with smallest element $0$ whose induced order $\le$ is given by $x\le y$ if $xy=x^2$ {\rm(}$x,y\in R${\rm)}.
\end{theorem}

\begin{proof}
Let $a,b,c\in R$. According to Proposition~\ref{prop2} (vi) and (1) we have
\begin{align*}
(a\wedge b)\wedge c & =\sum_{i=1}^{q-1}(a\wedge b)^ic^{q-i}=\sum_{i=1}^{q-1}\big(a^i+a^i(a+b)^{q-1}\big)c^{q-i}= \\
                    & =                  \sum_{i=1}^{q-1}a^ic^{q-i}+(a+b)^{q-1}\sum_{i=1}^{q-1}a^ic^{q-i}=(a\wedge c)+(a+b)^{q-1}(a\wedge c)= \\
                    & =(a\wedge c)+\left(\sum_{j=0}^{q-1}a^jb^{q-1-j}\right)\left(\sum_{k=1}^{q-1}a^kc^{q-k}\right)= \\
                    & =(a\wedge c)+\sum_{j=0}^{q-1}b^{q-1-j}\sum_{k=1}^{q-1}a^{j+k}c^{q-k}= \\
                    & =(a\wedge c)+\sum_{j=0}^{q-1}b^{q-1-j}\left(\sum_{k=1}^{q-1-j}a^{j+k}c^{q-k}+\sum_{k=q-j}^{q-1}a^{j+k}c^{q-k}\right)= \\
                    & =(a\wedge c)+\sum_{j=0}^{q-1}b^{q-1-j}\left(a^jc^j\sum_{k=1}^{q-j-1}a^kc^{q-j-k}+\sum_{k=q-j}^{q-1}a^{j+k-q+1}c^{q-k}\right)= \\
                    & =(a\wedge c)+\sum_{j=0}^{q-1}b^{q-1-j}\left(a^jc^j\sum_{k=1}^{q-j-1}a^kc^{q-j-k}+\sum_{m=1}^ja^mc^{j+1-m}\right)
\end{align*}
which is symmetric in $a$ and $c$. Hence we have
\[
(a\wedge b)\wedge c=(c\wedge b)\wedge a=a\wedge(b\wedge c).
\]
That $(R,\wedge)$ is a meet-semilattice with smallest element $0$ now follows by applying (i) of Proposition~\ref{prop2}. Finally, the following are equivalent: $a\le b$, $a\wedge b=a$, $a+a(a+b)^{q-1}=a$, $a(a+b)^{q-1}=0$, $a(a+b)=0$, $a^2+ab=0$, $ab=a^2$.	
\end{proof}

In the following let $\mathbb F_q=(F_q,+,\cdot)$ denote the $q$-element field. For the theory of finite fields cf.\ the monograph \cite{LN}.

\begin{lemma}
Let $a,b,c\in R$. Then the following hold:
\begin{enumerate}[{\rm(i)}]
\item If $\mathbf R=\mathbb F_q$ then $a\le b$ if and only if $a\in\{0,b\}$,
\item if $\mathbf R=(\mathbb F_q)^I$ then $(a_i;i\in I)\le(b_i;i\in I)$ if and only if for all $i\in I$ we have $a_i\in\{0,b_i\}$,
\item if $\mathbf R$ is unitary and $a\in R\setminus\{1\}$ then $a$ is not the greatest element of $(R,\le)$,
\item if $\mathbf R$ is unitary then $1$ is the greatest element of $(R,\le)$ if and only if $\mathbf R$ satisfies the identity $x^2\approx x$,
\item $a\le b$ implies $ac\le bc$,
\item if $a\le b$ then $a+c\le b+c$ if and only if $ac=bc$,
\item if $\mathbf R$ is unitary and $a\le b$ then $b+1\le a+1$ if and only if $(a+b)^2=a+b$.
\end{enumerate}	
\end{lemma}

\begin{proof}
\begin{enumerate}[(i)]
\item If $\mathbf R=\mathbb F_q$ then the following are equivalent: $a\le b$, $ab=a^2$, $a^2+ab=0$, $a(a+b)=0$, $0\in\{a,a+b\}$, $a=0$ or $a=b$, $a\in\{0,b\}$.
\item follows from (i).
\item If $\mathbf R$ is unitary then the following are equivalent: $a+1\le a$, $(a+1)a=(a+1)^2$, $a^2+a=a^2+1$, $a=1$.
\item If $\mathbf R$ is unitary then the following are equivalent: $a\le1$, $a\cdot1=a^2$, $a^2=a$.
\item Any of the following statements implies the next one: $a\le b$, $ab=a^2$, $(ac)(bc)=abc^2=a^2c^2=(ac)^2$, $ac\le bc$.
\item Under the assumption $a\le b$ the following are equivalent: $a+c\le b+c$, $(a+c)(b+c)=(a+c)^2$, $ab+ac+bc+c^2=a^2+c^2$, $ac+bc=0$, $ac=bc$.
\item Under the assumption $a\le b$ the following are equivalent: $b+1\le a+1$, $(b+1)(a+1)=(b+1)^2$, $ab+a+b+1=b^2+1$, $a^2+a+b=b^2$, $a^2+b^2=a+b$, $(a+b)^2=a+b$.
\end{enumerate}
\end{proof}

From (i) we see that in the meet-semilattice $(F_q,\le)$ the non-zero elements form an antichain.

Let $\mathcal V_n$ denote the variety of unitary commutative rings of characteristic $2$ satisfying the identity $x^{2^n}\approx x$.

The term operation from Definition~\ref{def1} is not distributive with respect to $+$ in general. E.g., for $n>1$, this term operation is not distributive with respect to $+$ in $\mathbb F_{2^n}$. Consequently, it is not a distributive term operation throughout the variety $\mathcal V_n$.

\begin{example}
Let $n>1$ and $a$ and $b$ be two distinct elements of $F_{2^n}\setminus\{0\}$. Then $0,a,b,a+b$ are pairwise distinct and
\[
(a+b)\wedge(a+b)=a+b\ne0=0+0=\big(a\wedge(a+b)\big)+\big(b\wedge(a+b)\big).
\]
For example, $\mathbb F_4\cong(\mathbb Z_2[x])/(x^2+x+1)$. We write $q(x)$ instead of $[q(x)]$. Now $x$ and $1$ are two distinct elements of $F_4\setminus\{0\}$, the operation table of $\wedge$ is as follows:
\[
\begin{array}{c|c|c|c|c}
\wedge & 0 & 1 & x & x+1 \\
\hline
   0   & 0 & 0 & 0 &  0 \\
\hline
   1   & 0 & 1 & 0 &  0 \\	
\hline
   x   & 0 & 0 & x &  0 \\	
\hline
  x+1  & 0 & 0 & 0 & x+1
\end{array}
\]
and we have
\[
(x+1)\wedge(x+1)=x+1\ne0=0+0=\big(x\wedge(x+1)\big)+\big(1\wedge(x+1)\big).
\]
\end{example}

\section{Derived lattice structure}

As we have seen above, the semilattice structure derived from the ring satisfying the identity $x^{2^n}\approx x$ is of a very special shape. The aim of this section is to involve another term operation for $\wedge$ which is distributive with respect to $+$ and, using Theorem~\ref{th1}, we obtain a way to derive a lattice structure which, in fact, will be a Boolean algebra.

A useful result for our next considerations is the following lemma.

\begin{lemma}\label{lem4}
Let $(R,+,\cdot)$ be a commutative ring of characteristic $2$ and $I$ a finite subset of $(\mathbb N_0)^2$, define a binary operation $\wedge$ on $R$ by
\[
x\wedge y:=\sum_{(i,j)\in I}x^{2^i}y^{2^j}
\]
for all $x,y\in R$ and assume $\wedge$ to be commutative. Then $\wedge$ is distributive with respect to $+$.
\end{lemma}

\begin{proof}
Using the remark preceding Theorem~\ref{th1} we have for $a,b,c\in R$
\begin{align*}
(a+b)\wedge c & =\sum_{(i,j)\in I}(a+b)^{2^i}c^{2^j}=\sum_{(i,j)\in I}(a^{2^i}+b^{2^i})c^{2^j}=\sum_{(i,j)\in I}a^{2^i}c^{2^j}+\sum_{(i,j)\in I}b^{2^i}c^{2^j}= \\
              & =(a\wedge c)+(b\wedge c).
\end{align*}	
\end{proof}

In what follows we use important results on finite fields (see \cite{LN}). First, we recall the concept of the {\em trace} of an element of a finite field.

For all $x\in F_q$ let
\[
\Tr(x):=\sum_{i=0}^{n-1}x^{2^i}
\]
denote the trace of $x$. Since according to the remark preceding Theorem~\ref{th1}
\[
\big(\Tr(x)\big)^2=(\sum_{i=0}^{n-1}x^{2^i})^2=\sum_{i=0}^{n-1}(x^{2^i})^2=\sum_{i=0}^{n-1}x^{2^{i+1}}=\sum_{i=0}^{n-1}x^{2^i}=\Tr(x)
\]
for all $x\in F_q$, we have $\Tr(x)\in\{0,1\}$ for all $x\in F_q$. It is well-known that the field $\mathbb F_q$ can be considered as an $n$-dimensional vector space over the two-element subfield $\mathbb F_2=(\{0,1\},+,\cdot)$ of $\mathbb F_q$. According to the well-known {\em Normal Basis Theorem} (cf.\ \cite{LN}, Theorem~2.35) there exists a so-called {\em normal basis} $B$ of $\mathbb F_q$, that is a basis of the form
\[
\{\alpha^{2^0},\dots,\alpha^{2^{n-1}}\}.
\]
Moreover, it is well-known (cf.\ \cite{LN}, Exercise~2.43) that there exists a {\em normal basis}
\[
\{\beta^{2^0},\dots,\beta^{2^{n-1}}\}
\]
of $\mathbb F_q$ that is {\em dual} to $B$ which means that
\[
\Tr(\alpha^{2^i}\beta^{2^j})=\delta_{ij}
\]
for all $i,j=0,\dots,n-1$.

For $i\in\mathbb N_0$ we abbreviate $\alpha^{2^i}$ by $\alpha_i$ and $\beta^{2^i}$ by $\beta_i$.

Now, we can construct the binary term operation $\wedge$. On $F_q$ we define a binary operation $\wedge$ by
\[
x\wedge y:=\sum_{i,j=0}^{n-1}\Tr(\alpha\beta_i\beta_j)x^{2^i}y^{2^j}
\]
for all $x,y\in F_q$. A more explicit form for $\wedge$ is as follows.

\begin{lemma}\label{lem1}
We have
\[
x\wedge y=x^{2^{n-1}}y^{2^{n-1}}+\sum_{\substack{i,j=0 \\ i<j}}^{n-1}\Tr(\alpha\beta_i\beta_j)(x^{2^i}y^{2^j}+x^{2^j}y^{2^i})=\sum_{k=0}^{n-1}\Tr(\beta_kx)\Tr(\beta_ky)\alpha_k
\]
for all $x,y\in F_q$.
\end{lemma}

\begin{proof}
The first formula follows from
\[
\Tr(\alpha\beta_i\beta_i)=\Tr(\alpha_0\beta_{i+1})=\delta_{i,n-1}
\]
for $i=0,\dots,n-1$. Concerning the second formula we have
\begin{align*}
x\wedge y & =\sum_{i,j=0}^{n-1}\Tr(\alpha\beta_i\beta_j)x^{2^i}y^{2^j}=\sum_{i,j=0}^{n-1}\sum_{k=0}^{n-1}(\alpha\beta_i\beta_j)^{2^k}x^{2^i}y^{2^j}=\sum_{k=0}^{n-1}\alpha_k\sum_{i=0}^{n-1}\beta_{i+k}x^{2^i}\sum_{j=0}^{n-1}\beta_{j+k}y^{2^j}= \\
          & =\sum_{k=0}^{n-1}\Tr(\beta_kx)\Tr(\beta_ky)\alpha_k
\end{align*}
for all $x,y\in F_q$.
\end{proof}

\begin{lemma}\label{lem5}
In the field $\mathbb F_q$ we have $\alpha_i\wedge\alpha_j=\delta_{ij}\alpha_i$ for all $i,j=0,\dots,n-1$.	
\end{lemma}

\begin{proof}
According to Lemma~\ref{lem1} we have for all $i,j=0,\dots,n-1$
\[
\alpha_i\wedge\alpha_j=\sum_{k=0}^{n-1}\Tr(\beta_k\alpha_i)\Tr(\beta_k\alpha_j)\alpha_k=\sum_{k=0}^{n-1}\delta_{ki}\delta_{kj}\alpha_k=\delta_{ij}\alpha_i.
\]
\end{proof}

For $I\subseteq\{0,\dots,n-1\}$ put
\[
x(I):=\sum_{i\in I}\alpha_i.
\]

\begin{lemma}\label{lem6}
Assume $I,J\subseteq\{0,\dots,n-1\}$. Then in $\mathbb F_q$ we have $x(\{0,\dots,n-1\})=1$ and $x(I)\wedge x(J)=x(I\cap J)$.	
\end{lemma}

\begin{proof}
Since $\alpha_0,\dots,\alpha_{n-1}$ are linearly independent, $\Tr(\alpha)$ cannot be equal to $0$. Therefore $\Tr(\alpha)=1$, i.e.\ $x(\{0,\dots,n-1\})=1$. Because of Lemma~\ref{lem1}, $\wedge$ is commutative and by Lemma~\ref{lem4}, it is distributive with respect to $+$. Now using Lemma~\ref{lem5} we obtain
\[
x(I)\wedge x(J)=(\sum_{i\in I}\alpha_i)\wedge(\sum_{j\in J}\alpha_j)=\sum_{(i,j)\in I\times J}(\alpha_i\wedge\alpha_j)=\sum_{(i,j)\in I\times J}\delta_{ij}\alpha_i=\sum_{k\in I\cap J}\alpha_k=x(I\cap J).
\]
\end{proof}

\begin{proposition}\label{prop1}
The operation $\wedge$ on $F_q$ is idempotent, commutative and associative and satisfies $x\wedge1=x$ for all $x\in F_q$.
\end{proposition}

\begin{proof}
Since $\{\alpha^{2^0},\dots,\alpha^{2^{n-1}}\}$ forms a basis of $\mathbb F_q$, every element of $F_q$ can be (uniquely) written in the form $x(I)$ for some subset $I$ of $\{0,\dots,n-1\}$. Using Lemma~\ref{lem6} we obtain
\begin{align*}
                     x(I)\wedge x(I) & =x(I\cap I)=x(I), \\
                     x(I)\wedge x(J) & =x(I\cap J)=x(J\cap I)=x(J)\wedge x(I), \\	
\big(x(I)\wedge x(J)\big)\wedge x(K) & =x(I\cap J)\wedge x(K)=x\big((I\cap J)\cap K\big)=x\big(I\cap(J\cap K)\big)= \\
                                     & =x(I)\wedge x(J\cap K)=x(I)\wedge\big(x(J)\wedge x(K)\big), \\
                         x(I)\wedge1 & =x(I)\wedge x(\{0,\dots,n-1\})=x(I\cap\{0,\dots,n-1\})=x(I)                                     
\end{align*}
for all $I,J,K\subseteq\{0,\dots,n-1\}$.
\end{proof}

For a given finite field $\mathbb F_q$ with $q=2^n$ we have now the following result.

\begin{theorem}\label{th2}
If for $x,y\in F_q$ we define
\begin{align*}
x\wedge y & :=\sum_{i,j=0}^{n-1}\Tr(\alpha\beta_i\beta_j)x^{2^i}y^{2^j}, \\
       x' & :=x+1, \\
  x\vee y & :=x+y+(x\wedge y)
\end{align*}
then $(F_q,\vee,\wedge,{}',0,1)$ is a Boolean algebra.
\end{theorem}

\begin{proof}
This follows from Theorem~\ref{th1}, Lemma~\ref{lem4} and Proposition~\ref{prop1}.
\end{proof}

In order to extend our result formulated in Theorem~\ref{th2} for fields to arbitrary unitary commutative rings of characteristic $2$ satisfying the identity $x^{2^n}\approx x$ for some $n\ge1$, we need the next proposition.

\begin{proposition}\label{prop3}
Let $q:=2^n$. Then $\mathcal V_n$ is generated by $\mathbb F_q$. More precisely,
\[
\mathcal V_n=\mathsf{SP}(\{\mathbb F_q\}).
\]
\end{proposition}

\begin{proof}
Let $\mathbf R=(R,+,\cdot,1)\in\mathcal V_n$. First, $\mathbf R$ is reduced. Indeed, if $a\in R$ is nilpotent, then $a^{q^m}=0$ for some sufficiently large positive integer $m$. On the other hand, the identity $x^q\approx x$
implies
\[
a^{q^m}=a.
\]
Hence $a=0$. Therefore, the nilradical of $\mathbf R$ is zero, and thus $\mathbf R$ embeds into the product of its prime quotients:
\[
R\hookrightarrow\prod_{\mathfrak p\in\operatorname{Spec}(\mathbf R)}\mathbf R/\mathfrak p.
\]
Each quotient $\mathbf R/\mathfrak p$ is an integral domain satisfying the identity $x^q\approx x$. Since every element of $\mathbf R/\mathfrak p$ is a root of the polynomial $X^q-X$, the quotient $\mathbf R/\mathfrak p$ has at most $q$ elements. Hence it is a finite field, say
\[
\mathbf R/\mathfrak p\cong\mathbb F_{2^d}
\]
for some positive integer $d$. \\
Now let $u\in(R/\mathfrak p)\setminus\{0\}$. Since $u^q=u$, we have
\[
u^{q-1}=1.
\]
Consequently, the exponent of the multiplicative group $(\mathbf R/\mathfrak p)^\times$, which is cyclic of order $2^d-1$ divides $q-1=2^n-1$. Therefore,
\[
2^d-1\mid 2^n-1.
\]
It follows that $d\mid n$. This can be seen as follows. There exist non-negative integers $k$ and $r$ such that $n=kd+r$ with $r<d$. Since $2^d\equiv1\mod2^d-1$, we obtain $2^n-1=2^{kd+r}-1\equiv2^r-1\mod2^d-1$. If $2^d-1\mid2^n-1$ then also $2^d-1\mid2^r-1$. Now $r>0$ would imply $0<2^r-1<2^d-1$ which is impossible. Hence $r=0$, and therefore $d\mid n$. Hence
\[
F_{2^d}\subseteq F_{2^n}=F_q.
\]
Thus every prime quotient $\mathbf R/\mathfrak p$ is isomorphic to a subfield of $\mathbb F_q$. Since $\mathbf R$ embeds into their direct product, we obtain
\[
\mathbf R\in\mathsf{SP}(\{\mathbb F_q\})
\]
proving $\mathcal V_n\subseteq\mathsf{SP}(\{\mathbb F_q\})$. \\
The converse inclusion is immediate, since $\mathbb F_q$ is a unitary commutative ring of characteristic $2$ satisfying the identity $x^q\approx x$, and these identities are preserved under subrings and direct products. Therefore
\[
\mathcal V_n=\mathsf{SP}(\{\mathbb F_q\}).
\]
\end{proof}

As a consequence we formulate our main result.

\begin{theorem}\label{th3}
Let $n\ge1$ and $\mathbf R=(R,+,\cdot,1)\in\mathcal V_n$. Then there exists a binary term operation $\wedge$ on $R$ such that $(R,\vee,\wedge,{}',0,1)$ is a Boolean algebra where $x\vee y:=x+y+(x\wedge y)$ and $x':=x+1$ for all $x,y\in R$.
\end{theorem}

\begin{proof}
According to Proposition~\ref{prop3}, $\mathbf R$ can be considered as a subalgebra of a direct product of fields for which the operation $\wedge$ is derived in Theorem~\ref{th2}. Hence we can consider now the operation $\wedge$ on $R$ componentwise. The remaining assertions follow immediately by Theorem~\ref{th1}.
\end{proof}

\section{Examples for $n\le5$}

\begin{example}
The case $n=1$ is well-known {\rm(}cf.\ e.g.\ {\rm\cite{Bi})} since we obtain a unitary Boolean ring. We have $\mathbb F_2\cong(\mathbb Z_2[x])/(x)$. If we choose $\alpha=\beta:=[1]$ then we obtain
\[
x\wedge y=xy.
\]
By Theorem~\ref{th3} and its proof, if $(R,+,\cdot,1)\in\mathcal V_1$ and
\begin{align*}
x\wedge y & :=xy, \\
  x\vee y & :=x+y+(x\wedge y), \\
       x' & :=x+1
\end{align*}
for all $x,y\in R$ then $(R,\vee,\wedge,{}',0,1)$ becomes a Boolean algebra.
\end{example}

\begin{example}
We have $\mathbb F_4\cong(\mathbb Z_2[x])/(x^2+x+1)$. If we choose $\alpha=\beta:=[x]$ then we obtain
\[
x\wedge y=xy^2+x^2y+x^2y^2.
\]
By Theorem~\ref{th3} and its proof, if $(R,+,\cdot,1)\in\mathcal V_2$ and
\begin{align*}
x\wedge y & :=xy^2+x^2y+x^2y^2, \\
  x\vee y & :=x+y+(x\wedge y), \\
       x' & :=x+1
\end{align*}
for all $x,y\in R$ then $(R,\vee,\wedge,{}',0,1)$ becomes a Boolean algebra. {\rm This result was already published in \cite{KDG}.}
\end{example}

\begin{example}
We have $\mathbb F_8\cong(\mathbb Z_2[x])/(x^3+x+1)$. If we choose $\alpha=\beta:=[x+1]$ then we obtain
\[
x\wedge y=xy^2+x^2y+x^2y^4+x^4y^2+x^4y^4.
\]
By Theorem~\ref{th3} and its proof, if $(R,+,\cdot,1)\in\mathcal V_3$ and
\begin{align*}
x\wedge y & :=xy^2+x^2y+x^2y^4+x^4y^2+x^4y^4, \\
  x\vee y & :=x+y+(x\wedge y), \\
       x' & :=x+1
\end{align*}
for all $x,y\in R$ then $(R,\vee,\wedge,{}',0,1)$ becomes a Boolean algebra. {\rm This result is already contained in the paper \cite{CK} by the first two authors.}
\end{example}

\begin{example}
We have $\mathbb F_{16}\cong(\mathbb Z_2[x])/(x^4+x+1)$. If we choose $\alpha:=[x^3+1]$ and $\beta:=[x^3+x^2+x+1]$ then we obtain
\[
x\wedge y=xy^4+x^4y+x^2y^4+x^4y^2+x^2y^8+x^8y^2+x^8y^8.
\]
By Theorem~\ref{th3} and its proof, if $(R,+,\cdot,1)\in\mathcal V_4$ and
\begin{align*}
x\wedge y & :=xy^4+x^4y+x^2y^4+x^4y^2+x^2y^8+x^8y^2+x^8y^8, \\
  x\vee y & :=x+y+(x\wedge y), \\
       x' & :=x+1
\end{align*}
for all $x,y\in R$ then $(R,\vee,\wedge,{}',0,1)$ becomes a Boolean algebra. {\rm This result is already contained in the paper \cite{CK} by the first two authors.}
\end{example}

\begin{example}
We have $\mathbb F_{32}\cong(\mathbb Z_2[x])/(x^5+x^2+1)$. If we choose $\alpha=\beta:=[x+1]$ then we obtain
\[
x\wedge y=xy^2+x^2y+x^2y^8+x^8y^2+x^4y^8+x^8y^4+x^4y^{16}+x^{16}y^4+x^{16}y^{16}.
\]
By Theorem~\ref{th3} and its proof, if $(R,+,\cdot,1)\in\mathcal V_5$ and
\begin{align*}
x\wedge y & :=xy^2+x^2y+x^2y^8+x^8y^2+x^4y^8+x^8y^4+x^4y^{16}+x^{16}y^4+x^{16}y^{16}, \\
  x\vee y & :=x+y+(x\wedge y), \\
       x' & :=x+1
\end{align*}
for all $x,y\in R$ then $(R,\vee,\wedge,{}',0,1)$ becomes a Boolean algebra. {\rm This result is new.}
\end{example}








Authors' addresses:

Ivan Chajda \\
Palack\'y University Olomouc \\
Faculty of Science \\
Department of Algebra and Geometry \\
17.\ listopadu 12 \\
771 46 Olomouc \\
Czech Republic \\
ivan.chajda@upol.cz

Miroslav Kola\v r\'ik \\
Palack\'y University Olomouc \\
Faculty of Science \\
Department of Computer Science \\
17.\ listopadu 12 \\
771 46 Olomouc \\
Czech Republic \\
miroslav.kolarik@upol.cz

Helmut L\"anger \\
TU Wien \\
Faculty of Mathematics and Geoinformation \\
Institute of Discrete Mathematics and Geometry \\
Wiedner Hauptstra\ss e 8-10 \\
1040 Vienna \\
Austria, and \\
Palack\'y University Olomouc \\
Faculty of Science \\
Department of Algebra and Geometry \\
17.\ listopadu 12 \\
771 46 Olomouc \\
Czech Republic \\
helmut.laenger@tuwien.ac.at
\end{document}